\documentclass[leqno,twoside,12pt]{amsart}
\usepackage{latexsym, esint, url, amssymb, yfonts, amsmath}
\usepackage[utf8]{inputenc}
\usepackage{graphicx,color}
\usepackage{tikz-cd}
\usepackage{caption}
\usepackage{subcaption}
\usepackage{array}
\newtheorem*{theorem*}{Theorem} 
\newtheorem{theorem}{Theorem}[section]
\newtheorem{lemma}[theorem]{Lemma}

\newtheorem{corollary}[theorem]{Corollary}
\newtheorem{definition}[theorem]{Definition}

\theoremstyle{definition}
\newtheorem{remark}[theorem]{Remark}
\newtheorem{example}[theorem]{Example}

\numberwithin{equation}{section}
\numberwithin{figure}{section}

\definecolor{Green}{rgb}{0, 0.65,0}

\begin{document}
		
\title{Boundary Regularity for the $\infty$-Heat Equation}
\author{Nikolai Ubostad}
\address{Department of Mathematical Sciences 
Norwegian University of Science and Technology 
N-7491 Trondheim, Norway}
\email{nikolai.ubostad@ntnu.no}

\begin{abstract}
We study the boundary regularity for the normalised $\infty$-heat equation $u_t =  \Delta_{\infty}^Nu$ in arbitrary domains. Perron's Method is used for constructing solutions. We characterize regular boundary points with barrier functions, and prove an Exterior Sphere condition. A Petrovsky criterion is established.
\end{abstract}
\maketitle
\section{Introduction}
\noindent The $p$-parabolic equation
\begin{equation}
\label{eq:p-parabolic}
u_t = \text{div}(|Du|^{p-2}Du),
\end{equation}
$1<p <\infty$,
has been extensively studied, see for example \cite{barenblatt} and \cite{dibenedetto2012degenerate}. The study of the Dirichlet problem on \emph{arbitrary} domains was initiated by Kilpeläinen and Lindqvist in \cite{kilpelainen1996dirichlet}, where the Perron Method was used to construct solutions.

Equations of this type are typically difficult to study in domains that are not space-time cylinders, i.e. in domains not of the form $Q \times (0, T) \subset \mathbb{R}^n\times (0, T)$.

In this paper, we study the \emph{normalised} $\infty$-parabolic equation, also called the \emph{normalised} $\infty$-heat equation,
\begin{equation}
\label{eq:dirichlet problem} 
\begin{cases}
u_t - \Delta_{\infty}^N u=0 \ &\text{in} \ \Omega, \\
u = f \ &\text{on} \ \partial \Omega,
\end{cases}
\end{equation}
where the normalised $\infty$-Laplace operator is defined by
\begin{equation}
\label{eq:infty laplace}
\Delta_{\infty}^Nu :=\bigg\langle D^2u\frac{Du}{|Du|},\frac{Du}{|Du|}\bigg\rangle = \frac{1}{|Du|^2} \sum_{i, j =1}^{n}u_{x_i}u_{x_j}u_{x_ix_j}.
\end{equation}
Note that if $u(x, t) =G(r, t)$, for a smooth function $G$ and $r=|x|$, \eqref{eq:infty laplace} reads $G_{rr}(r, t)$ when $x\neq 0$.

The initial value problem was studied by Juutinen and Kawohl in \cite{Juutinen2006}. They proved a comparison principle, and the existence of viscosity solutions in space-time cylinders using an approximation argument. As in the $p$-parabolic case, arbitrary domains are more challenging. Following Kilpeläinen and Lindqvist, we employ Perron's method to construct solutions in arbitrary domains $\Omega$ in $\mathbb{R}^n\times (-\infty, \infty)$. We shall prescribe boundary values on the whole boundary $\partial \Omega$ as if \eqref{eq:dirichlet problem} were an elliptic equation. The boundary values are in general not attained at all points.

These Perron solutions are always viscosity solutions of \eqref{eq:dirichlet problem} in $\Omega$. However, the values they takes on the boundary may not be the same as the desired boundary values $f(\zeta)$. Points satisfying 
\[
\lim_{\eta \to \zeta}u(\eta)=f(\zeta), \ \eta \in \Omega, \ \zeta \in \partial \Omega
\]
 for \emph{every} continuous boundary function $f$ are called \emph{regular points} of the boundary for the $\infty$-heat equation.  A point $\zeta$ of the boundary satisfies a \emph{barrier condition} if there exists a supersolution $w$ defined on the entire domain, such that $w(\zeta)=0$ and $w(\eta)>0$ for $\zeta \neq \eta$. In this case, $w$ is called a \emph{barrier function} in $\Omega$.

The stationary equation
\begin{equation}
\label{eq:stationaryeg}
\sum_{i, j =1}^{n}u_{x_i}u_{x_j}u_{x_ix_j}=0
\end{equation}
has an interesting story. It was first studied by Aronsson in \cite{aronsson} in connection with absolutely minimizing Lipchitz extensions of functions. Because of this property, all boundary points are regular for \eqref{eq:stationaryeg}.  Other applications are image processing, see  \cite{caselles}, and it arises in connection with stochastical games, cf.  \cite{peres}. We mention that there also is a game-theoretic interpretation of \eqref{eq:dirichlet problem}, see \cite{manfredi2010asymptotic}.

Our main results are as follows.

\begin{theorem}
\label{thm:barrier}
A boundary point $\zeta_0$ is regular for \eqref{eq:dirichlet problem} if and only if there exists a barrier at $\zeta_0,$
\end{theorem}
This is a criterion for regular points on the boundary. Using this, we prove that every point on the parabolic boundary of cylinders are regular:
\begin{theorem}
\label{thm:regularcylinders}
Let $Q \subset \mathbb{R}^n$, and let $Q_T= Q \times (0, T)$. Assume $0<t_0<T$ and $x_0 \in \partial Q$. Then $\zeta_0 =(x_0, t_0)$ is a regular point of $\partial_p Q_T$.
\end{theorem}
 We also have the following Petrovsky condition
 \begin{theorem}
 \label{thm:petrowski}
 The origin $(x, t)=(0, 0)$ is a regular point for \eqref{eq:dirichlet problem} in the domain enclosed by the hypersurfaces 
\begin{equation*}
\{(x, t)\in \mathbb{R}^n\times (0, \infty)\ : \ |x|^2 =-4t\log|\log|t||\} \ \text{and} \ \{t=-c\}, 
\end{equation*}
for $0<c<1$.
 \end{theorem}
 Note that this result is analogous to the 1-dimensional heat equation, see \cite{petrovsky1935ersten}. However, Theorem \ref{thm:petrowski} is completely independent of  the number of spatial variables, in contrast to the same result for the heat equation. 
 
  Theorem \ref{thm:petrowski} is in some sense sharp:
 \begin{theorem}
  \label{thm:not regular}
  The origin is \emph{not} a regular point for the domain defined by $|x|^2 =-4(1+\epsilon)t\log{|\log{|t|}|}$ for any $\epsilon >0$.
  \end{theorem}
These results are very similar to the theorems for the $p$-parabolic equation in \cite{Lindqvist1995}. Our Petrovsky condition is somewhat easier, due to the similarity of \eqref{eq:dirichlet problem} and the heat equation.

We limit our presentation to bounded domains and continuous boundary functions.

Our work is structured as follows. Section \ref{sec:existence} contains the existence and comparison results, while Section \ref{sec:perron} introduces the Perron method for viscosity solutions of the equation. In Section \ref{sec:barrier} we introduce the concept of barrier functions, and we provide an Exterior Sphere condition. We also prove that the future cannot influence the present. A brief discussion of the so-called Infinity Heat Balls is also provided. The proof of Theorem \ref{thm:regularcylinders} is provided in Section \ref{sec:regularcylinders}. A Petrovsky criterion is established in Section \ref{sec:petrowski}. Appendix \ref{sec:appa} contains a brief discussion of the boundary regularity for the related equation
\begin{equation*}
u_t = \sum_{i, j =1}^{n}u_{x_i}u_{x_j}u_{x_ix_j},
\end{equation*}
studied by Crandall and Wang in \cite{crandall2003another}. In this case we do not know if it suffices with a single barrier. We prove that if there exists a \emph{family} of barriers at $\zeta_0$, then $\zeta_0$ is a regular point.

\subsection{Notation}

In what follows $\Omega$ is an \emph{arbitrary} bounded domain in $\mathbb{R}^n \times (-\infty, \infty)$. $Q_T$ is a space-time cylinder: $Q_T = Q\times (0, T)$, $\partial \Omega$ is the Euclidean boundary of $\Omega$ and $\partial_pQ_T$ is the $\emph{parabolic}$ boundary of $Q_T$, i.e. $(\overline{Q}\times \{0\}) \cup (\partial Q\times (0, T])$. (It consists of the "bottom" and the sides of the cylinder. The top is excluded.) $\zeta, \eta \in \mathbb{R}^n \times \mathbb{R}$ are points in space-time, that is $\zeta =(x, t)$.

$Du$ denotes the gradient with respect to the spatial coordinates $x$ of $u$, and $D^2u$ is the spatial Hessian matrix of $u$.

 The space of symmetric $n\times n$ matrices is denoted by $\mathbb{S}^n$. The diameter of a set $\Omega \subset \mathbb{R}^{n+1}$ is defined by
\[
\text{diam}(\Omega) = \sup_{\zeta, \eta \in \Omega}|\zeta-\eta|.
\] 
The space of lower  semicontinuous functions from $\Omega$ to \\ $\mathbb{R}\cup \{\infty\}$ is denoted by $\operatorname{LSC}(\Omega)$, while $\operatorname{USC}(\Omega)$ are the upper semicontinuous ones.
\section{Existence and Comparison}
\label{sec:existence}
\noindent We seek to construct solutions to the $\infty$- heat equation
\begin{equation*}
\begin{cases}
u_t - \Delta_{\infty}^N u=0 \ &\text{in} \ \Omega, \\
u = f \ &\text{on} \ \partial \Omega,
\end{cases}
\end{equation*}
using Perron's method. The following definitions are taken from \\ \cite{dibenedetto1995}:
\begin{definition}
	\label{def:viscositydefinition}
We say that $u \in \operatorname{USC}(\Omega)$ is a \emph{viscosity subsolution} to \eqref{eq:dirichlet problem} if, for every function $\psi \in C^2(\Omega)$ such that $u-\psi$ has a maximum at $\zeta_0$, we have
\begin{equation*}
\begin{cases}
\psi_t(\zeta_0)-\Delta_{\infty}^N \psi(\zeta_0)  \leq 0, \ \text{for} \ D\psi(\zeta_0) \neq 0 \\
\psi_t(\zeta_0) -\Lambda(D^2\psi(\zeta_0)) \leq 0, \ \text{for} \ D\psi(\zeta_0) =0.
\end{cases}
\end{equation*}
Likewise, $v \in \operatorname{LSC}(\Omega) $ is a \emph{viscosity supersolution} in $\Omega$ if, for every $\psi \in C^2(\Omega)$ such that $v-\psi$ has a minimum at $\zeta_0$, we have
\begin{equation*}
\begin{cases}
\psi_t(\zeta_0)-\Delta_{\infty}^N \psi(\zeta_0)  \geq 0, \ \text{for} \ D\psi(\zeta_0) \neq 0 \\
\psi_t(\zeta_0) -\lambda(D^2\psi(\zeta_0)) \geq 0, \ \text{for} \ D\psi(\zeta_0) =0,
\end{cases}
\end{equation*}
where $\lambda(D^2\psi(\zeta_0)), \Lambda(D^2\psi(\zeta_0))$ is the minimal and maximal eigenvalue respectively of the Hessian matrix $D^2\psi(\zeta_0)$.
A function that is both a viscosity sub- and supersolution is a \emph{viscosity solution}.
\end{definition}
\begin{remark}
	An important result in \cite{crandall1992user} is that if a function $u\in C^2(\Omega)$ is such that $Du \neq 0$ and $u_t-\Delta^N_{\infty}u=0$ holds pointwise, then $u$ is a viscosity solution of \eqref{eq:dirichlet problem}. Hence we can calculate $u_t-\Delta^N_{\infty}u$ for concrete $C^2$- functions to verify that they are valid viscosity solutions. If $Du= 0$, more care is needed.
\end{remark}
Equivalently, we can characterize viscosity solutions in terms of the parabolic semi-jets.
\begin{definition}
The parabolic upper semi-jet of $u$ at $(x, t)$ is denoted by $\mathcal{P}^{2, +}_{\Omega}u(x, t)$ and is the set of all $(a, p, X) \in \mathbb{R}\times \mathbb{R}^n\times\mathbb{S}^n$ so that
\begin{equation*}
u(x, t) \leq a(t-s) +\langle p, x-y \rangle+\frac{1}{2}\langle X(x-y), x-y \rangle + o((t-s)+|x-y|^2),
\end{equation*}
as $(y, s) \to (x, t)$. Furthermore, $\mathcal{P}^{2, -}_{\Omega}u(t, x)= -\mathcal{P}^{2, +}_{\Omega}(-u(x, t))$.
\end{definition}
One can then prove the following equivalent definition, see Chapter 8 of \cite{crandall1992user}:
\begin{theorem}
\label{thm:equivalentdef}
A viscosity subsolution of \eqref{eq:dirichlet problem} is a function $u \in \operatorname{USC(\Omega)}$ so that
\begin{equation*}
\begin{cases}
a-\bigg\langle X\frac{p}{|p|},\frac{p}{|p|}\bigg\rangle \leq 0 , \ &\text{if} \ p \neq 0 \\
a-\Lambda(X) \leq 0 \ &\text{if} \ p=0,
\end{cases}
\end{equation*}
for $(a, p, X) \in \mathcal{P}^{2, +}_{\Omega}u(x, t)$. A similar result holds for viscosity supersolutions.
\end{theorem}
In what follows, a (sub/super-) solution will always refer to a viscosity (sub/super-) solution.
\begin{remark}
	in the literature the terms viscosity supersolutions and $p$-superparabolic functions appear in connection with equation \eqref{eq:p-parabolic}. The latter ones are defined via a comparison principle with respect to the weak solutions. Both concepts define the same class of functions. For the equation \eqref{eq:dirichlet problem}, there are no weak solutions to compare with, only viscosity solutions are available. Thus we avoid the term superparabolic here.
\end{remark}

 We need some preliminary results regarding existence and comparison, both from Juutinen \& Kawohl. The first result is Theorem 3.1 in  \cite{Juutinen2006}. 
 
 \begin{theorem}[Comparison]
 \label{thm:comparison}
 Let $Q_T = Q\times (0, T)$, where $Q$ is a bounded domain, and let $u$, $v$ be, respectively, a supersolution and a subsolution of the equation in $Q_T$. Assume further that
 \begin{equation*}
 \limsup_{\zeta \to \zeta_0} u(\zeta) \leq \liminf_{\zeta \to \zeta_0} v(\zeta) 
 \end{equation*}
 for all $\zeta_0 \in \partial_pQ_T$ (excluding the possibility that both sides are either $\infty$ or $-\infty$ at the same time). Then $u \leq v$ in $Q_T$.
 \end{theorem}
 The following existence result is Theorem 4.1 in  \cite{Juutinen2006}.
\begin{theorem}[Existence]
\label{thm:existence}
Let $Q_T = Q\times (0, T)$, where $Q$ is a bounded domain, and let $\phi\in C(\mathbb{R}^{n+1})$. Then there exists a unique $h \in C(Q_T \cap \partial_pQ_T)$ such that $h=\phi $ on $\partial_pQ_T$ and
\begin{equation*}
h_t - \Delta_{\infty}^N h=0 \ \text{in} \ Q_T
\end{equation*}
\end{theorem}
We wish to extend both theorems to an arbitrary domain.

The following lemma will be useful to us:
\begin{lemma}
\label{thm:sequence}
Let $u_i$ be a  locally uniformly bounded sequence of viscosity solutions in $\Omega$. Then there exists a subsequence that converges locally uniformly in $\Omega$ to a viscosity solution.
\end{lemma}
\begin{proof}
Let $\Sigma \subset \Omega$ be a proper subdomain. Corollary 5.2 in Juutinen \& Kawohl guarantees the existence of a constant
$C>0$ so that 
\[
|Du_i(\zeta)| \leq C\left(1+\frac{||u_i||_{\infty}}{\text{dist}(\zeta, \partial_p\Omega)} \right),
\]
that is, $\{u_i\}$ is equicontinuous in $\Sigma$. The Arzelà-Ascoli theorem gives the existence of a subsequence converging uniformly in $\Sigma$ to a continuous function. Exhausting $\Omega$ by a sequence of subdomains and using an diagonalization argument, we obtain a locally uniformly convergence sequence, again denoted by $u_i$.

It remains to show that $u$ is a solution on each box $Q_{t_1, t_2}=Q\times(t_1, t_2) \subset \Omega$. Let $h$ be the solution in $Q_{t_1, t_2}$ with boundary values $u$ on $\partial_pQ_{t_1, t_2}$ (Theorem \ref{thm:existence} guarantees the existence of such a $h$.) Let $\epsilon >0$. For large enough values of $i$, we have
\begin{equation*}
h-\epsilon = u-\epsilon < u_i < u+\epsilon =h+\epsilon
\end{equation*}
on $\partial_pQ_{t_1, t_2}$. The comparison principle \ref{thm:comparison}  gives
\begin{equation*}
h-\epsilon \leq u_i \leq h+\epsilon
\end{equation*}
in $Q_{t_1, t_2}$, and so
\begin{equation*}
h-\epsilon \leq u \leq h+\epsilon,
\end{equation*}
and $u=h.$
\end{proof}

%
%
%
%
%
We shall prove a comparison result for arbitrary domains. For this we rely on the parabolic version of Ishii's Lemma, taken from \cite{crandall1990maximum}.
\begin{theorem}
Let $u, v \in \operatorname{USC}(Q \times Q \times (0, T))$, be  subsolutions to the $\infty$-heat equation. Let $\theta(x, y, t)$ be once differentiable in $t$ and twice differentiable in $(x, y)$. Define
\begin{equation*}
w(x, y, t) = u(x, t)+v(y, t)-\theta(x, y, t),
\end{equation*}
and assume that $w$ attains an interior maximum at the point $(\hat{x}, \hat{y}, s)\in Q \times Q \times (0, T).$ Then, for each $\epsilon >0$ there are matrices $X, Y \in \mathbb{S}^n$ such that
\begin{equation*}
(b_1, D_x\theta(\hat{x}, \hat{y}, s), X) \in \bar{\mathcal{P}}^{2, +}_{\Omega}u(\hat{x}, s),
\end{equation*} 
\begin{equation*}
(b_2, D_y\theta(\hat{x}, \hat{y}, s), Y) \in \bar{\mathcal{P}}^{2, +}_{\Omega}v(\hat{y}, s)
\end{equation*} 
\begin{equation*}
-\left(\frac{1}{\epsilon}+||A||\right)I \leq
 \begin{pmatrix}
X & 0 \\
0 & Y
 \end{pmatrix}
 \leq
 A+\epsilon A^2
\end{equation*}
and $b_1+b_2 =\theta_t(\hat{x}, \hat{y}, s)$ and $A=D^2\theta(\hat{x}, \hat{y}, s)$.
\end{theorem}
We can now prove the comparison theorem for general domains.
\begin{theorem}
	\label{thm:comparison_general}
Let $\Omega$ be bounded domain in $\mathbb{R}^n\times\mathbb{R}$, and let $u$, be a subsolution and $v$ a supersolution in $\Omega$. Assume further that
\begin{equation*}
\limsup_{\zeta \to \zeta_0} u(\zeta) \leq \liminf_{\zeta \to \zeta_0} v(\zeta) 
\end{equation*}
for all $\zeta_0 \in \partial\Omega$ (excluding the possibility that both sides are either $\infty$ or $-\infty$ at the same time). Then $u \leq v$ in $\Omega$.
\end{theorem}
\begin{proof}
We make the following \emph{antithesis}:

Assume that 
\begin{equation}
\sup_{\Omega}(u(x, t)-v(x, t)) >0
\label{eq:antithesis}
\end{equation}
and that the supremum is attained at some $(\hat{x}, \hat{t})\in \Omega$.
Consider the function
\begin{equation}
w_{\epsilon}(x, t, y, s) =u(x, t)-v(y, s) -\frac{1}{4\epsilon}|x-y|^4-\frac{1}{2\epsilon}(t-s)^2,
\end{equation}
and let $w_{\epsilon}$ have a maximum at $(x_{\epsilon}, t_{\epsilon}, y_{\epsilon}, s_{\epsilon})$. Since we assumed that
\[
 \limsup_{\zeta \to \zeta_0} u(\zeta) \leq \liminf_{\zeta \to \zeta_0} v(\zeta), 
\]
 \eqref{eq:antithesis} implies that this maximum is in the interior of $\Omega \times \Omega$ for $\epsilon$ small enough, see Theorem 3.4 in \cite{koike}.

{\bf 1.} If $x_{\epsilon}=y_{\epsilon}$, then $v-\phi$ has a minimum at $(y_{\epsilon}, s_{\epsilon})$, where
\[
\phi(y, s) =-\frac{1}{4\epsilon}|x-y|^4-\frac{1}{2\epsilon}(t-s)^2.
\]
 Since $v$ is a  supersolution and $D\phi(y_{\epsilon}, s_{\epsilon})=0$, we get
\[
0 <\phi_t(y_{\epsilon}, s_{\epsilon}) <\lambda(D^2\phi(y_{\epsilon}, s_{\epsilon}) =\frac{1}{\epsilon}(t_{\epsilon}-s_{\epsilon}).
\]
In the same way we get that $u-\psi$, where 
\[
\psi(x, t) =\frac{1}{4\epsilon}|x-y|^4+\frac{1}{2\epsilon}(t-s)^2,
\]
has a maximum at $(x_{\epsilon}, t_{\epsilon})$. Thus
\[
0>\psi_t(x_{\epsilon}, t_{\epsilon}) >\Lambda(D^2\phi(x_{\epsilon}, t_{\epsilon}))=\frac{1}{\epsilon}(t_{\epsilon}-s_{\epsilon}).
\]
Subtracting these two inequalities gives
\[
0<\frac{1}{\epsilon}(t_{\epsilon}-s_{\epsilon})-\frac{1}{\epsilon}(t_{\epsilon}-s_{\epsilon})=0,
\]
a contradiction.

{\bf 2.} If $x_{\epsilon}\neq y_{\epsilon}$, Ishii's Lemma now gives, after fixing $t=t_{\epsilon}$ and $s= s_{\epsilon}$ in $w_{\epsilon}$, the existence of symmetric matrices $X_{\epsilon}, Y_{\epsilon}$ so that

\begin{align}
&\left(\frac{1}{\epsilon}(t_{\epsilon}-s_{\epsilon}), \frac{1}{\epsilon}|x_{\epsilon}-y_{\epsilon}|(x_{\epsilon}-y_{\epsilon}), X_{\epsilon}\right) \in \bar{\mathcal{P}}^{2, +}u(t_{\epsilon}, x_{\epsilon}), \label{eq:onejet} \\
&\left(\frac{1}{\epsilon}(t_{\epsilon}-s_{\epsilon}), \frac{1}{\epsilon}|x_{\epsilon}-y_{\epsilon}|(x_{\epsilon}-y_{\epsilon}), Y_{\epsilon}\right) \in \bar{\mathcal{P}}^{2, -}v(s_{\epsilon}, y_{\epsilon}), \label{eq:twojet}
\end{align}
and $Y_{\epsilon} \geq X_{\epsilon}$. Since $u$ is a subsolution and $v$ is a supersolution, we get by Theorem \ref{thm:equivalentdef} and  inserting \eqref{eq:onejet} and \eqref{eq:twojet} into the equation \eqref{eq:dirichlet problem} and subtracting, 
\begin{align*}
0< \frac{1}{\epsilon} (t_{\epsilon}-s_{\epsilon}) -\left(Y_{\epsilon}\frac{(x_{\epsilon}-y_{\epsilon})}{|x_{\epsilon}-y_{\epsilon}|}\right)\cdot \frac{(x_{\epsilon}-y_{\epsilon})}{|x_{\epsilon}-y_{\epsilon}|} \\
-\left[ \frac{1}{\epsilon} (t_{\epsilon}-s_{\epsilon}) -\left(X_{\epsilon}\frac{(x_{\epsilon}-y_{\epsilon})}{|x_{\epsilon}-y_{\epsilon}|}\right)\cdot \frac{(x_{\epsilon}-y_{\epsilon})}{|x_{\epsilon}-y_{\epsilon}|}\right] \\
= -\left((Y_{\epsilon}-X_{\epsilon})\frac{(x_{\epsilon}-y_{\epsilon})}{|x_{\epsilon}-y_{\epsilon}|}\right)\cdot \frac{(x_{\epsilon}-y_{\epsilon})}{|x_{\epsilon}-y_{\epsilon}|} \leq 0,
\end{align*}
a contradiction.
\end{proof}
\section{The Perron Method}
\label{sec:perron}
\noindent We start  with a definition. Let $f : \partial \Omega \to \mathbb{R}$ be a bounded function.
\begin{definition}
A function $u$ belongs to the \emph{upper class} $\mathcal{U}_f$ if $u$ is a viscosity supersolution in $\Omega$ and
\begin{equation*}
\liminf_{\eta \to \zeta}u(\eta) \geq f(\zeta)
\end{equation*}
for $\zeta \in \partial \Omega$.

A function $v$ belongs to the \emph{lower class} $\mathcal{L}_f$ if $v$ is a viscosity sub solution in $\Omega$ and
\begin{equation*}
\limsup_{\eta \to \zeta}v(\eta) \leq f(\zeta)
\end{equation*}
for $\zeta \in \partial \Omega$.

We define the \emph{upper solution} 
\[
\overline{H}_f(\zeta)=\inf\{u(\zeta) : u \in \mathcal{U}_f \},
\] 
and the \emph{lower solution}
\[
\underline{H}_f(\zeta)=\sup\{v(\zeta) : v \in \mathcal{L}_f \}.
\]
 Note that the $\inf$ and $\sup$ are taken over \emph{functions}.
 \end{definition}
\begin{remark}
The Comparison Principle \ref{thm:comparison_general} gives immediately that $v \leq u$ in $\Omega$, for $v \in\mathcal{L}_f$ and $u \in \mathcal{U}_f$, and hence 
\[\underline{H}_f\leq \overline{H}_f.\]
 Whether $\underline{H}_f= \overline{H}_f$ holds in general, is a more subtle question.
\end{remark}
A classical tool in the potential theory is the \emph{parabolic modification}: Let $Q_T=Q\times[0, T]$ be a cylinder contained in $\Omega$, and define 
\[
v(\zeta)= \sup\{h(\zeta) : h \in C(Q_T), h \ \text{is} \ \infty  -\ \text{parabolic and} \ h \leq u \ \text{on} \ \partial_p Q_T \}.
\]
Now let
\begin{equation*}
U= \begin{cases}
u \ \text{in} \ \Omega \setminus Q_T \\
v \ \text{in} \ Q_T.
\end{cases}
\end{equation*}
It is clear that $U \leq u$ in $\Omega$. Furthermore, we have the following
\begin{lemma}
U is a supersolution in $\Omega$ and a solution in $Q_T$.
\end{lemma}
\begin{proof}
Choose an increasing sequence $h_i$ of continuous functions on $\partial_pQ_T$ so that $\lim_{i \to \infty}h_i =u$, and  let $u_i$ be the solution with boundary values $h_i$. The Comparison Principle \ref{thm:comparison} now implies that $u_i$ is an increasing sequence, and that the limit function is $v$. Since $u_i$ is bounded, Lemma \ref{thm:sequence} gives that the limit function is a solution. The immediate consequence is that $U$ is a supersolution in $\Omega.$
\end{proof}
\subsection{Perron solutions are viscosity solutions}
\label{sec:perron=viscosity}
 We show that the Upper Perron solution is a viscosity solution. 
 
Letting $\overline{H}=\overline{H}_f$, we show that $H$ is a supersolution.We check that 
\begin{equation*}
\begin{cases}
\phi_t(\zeta_0)-\Delta_{\infty}^N\phi(\zeta_0) \geq 0 \ \text{when} \ D\phi(\zeta_0) \neq 0, \\
\phi_t(\zeta_0)-\lambda(D^2\phi(\zeta_0)) \geq 0 \ \text{when} \ D\phi(\zeta_0) = 0,
\end{cases}
\end{equation*}
whenever $\phi \in C^2(\Omega)$ and 
\[
0=(\overline{H}-\phi)(\zeta_0)<(\overline{H}-\phi)(\zeta),
\]
 for $\zeta \in \Omega \setminus \{\zeta_0\}$.

Let $r>0$ be so that $B_{2r}(\zeta_0)\subset \Omega$. By our assumption, we can find a $\sigma>0$ so that
\begin{equation}
\label{eq:inf on ball}
\inf_{\partial B_{r}(\zeta_0)}(\overline{H}-\phi) = \sigma
\end{equation}
Now choose a sequence of points $\{\zeta_{\epsilon}\}\subset B_{r}$ so that $\zeta_{\epsilon}\to \zeta_0$ and $|\phi(\zeta_{\epsilon})-\phi(\zeta_0)|<\epsilon$, as $\epsilon \to 0$. Also choose $\{u_{\epsilon}\} \subset \mathcal{U}_f$ so that
\begin{equation*}
u_{\epsilon}(\zeta_{\epsilon})+\epsilon \geq \overline{H}(\zeta_0).
\end{equation*}
Equation \eqref{eq:inf on ball} gives that for small $\epsilon<\sigma/2$, we have
\begin{equation*}
\inf_{\partial B_{r}(\zeta_0)}(u_{\epsilon}-\phi) > (u_{\epsilon}-\phi)(\zeta_{\epsilon}).
\end{equation*}
Thus there are points $\eta_{\epsilon} \in B_r(\zeta_0)$ where $u_{\epsilon}-\phi$ attains its minimum over $B_r(\zeta_0)$.

{\bf 1.} If $D\phi(\eta_{\epsilon})\neq 0$,  by definition we have
\begin{equation}
\label{eq:eta super}
\phi_t(\eta_{\epsilon}) -\Delta_{\infty}^N\phi(\eta_{\epsilon})\geq 0
\end{equation}
since $u_{\epsilon}$ is a supersolution. We can assume that $\eta_{\epsilon} \to \eta_0$ (possibly along a subsequence). Then we get
\begin{align*}
&(\overline{H}-\phi)(\zeta_0) 
\geq (u_{\epsilon}-\phi)(\zeta_{\epsilon})+2\epsilon \\
&\geq (u_{\epsilon}-\phi)(\eta_{\epsilon})+2\epsilon 
\geq (\overline{H}-\phi)(\eta_{\epsilon})+2\epsilon 
\to (\overline{H}-\phi)(\eta_0),
\end{align*}
and hence $\eta_0 = \zeta_0$, and sending $\epsilon \to 0$ in \eqref{eq:eta super}, we get \\
$\phi_t(\zeta_0)-\Delta_{\infty}^N\phi(\zeta_0) \geq 0$.

{\bf 2.} If $D\phi(\eta_{\epsilon})= 0$, by definition we have
\[
\phi_t(\eta_{\epsilon})-\lambda(D^2\phi(\eta_{\epsilon})) \geq 0.
\]
Arguing as above, sending $\epsilon \to 0$, 
\[
\phi_t(\zeta_0)-\lambda(D^2\phi(\zeta_0)) \geq 0,
\]
by the continuity of the mapping $\zeta \to\lambda(D^2\phi(\zeta))$. Thus the Upper Perron solution $\overline{H}$ is a viscosity supersolution.

Theorem 4.3 in \cite{koike}, together with the comparison principle \ref{thm:comparison_general} gives that $\overline{H}$ is a subsolution, as well.

 The proof that the Lower Perron solution is a viscosity solution is similar.
%
%

\section{Barrier functions}
\label{sec:barrier}
\begin{definition}
We say that $\zeta_0 \in \partial \Omega $ is \emph{regular} if 
\begin{equation*}
\lim_{\zeta \to \zeta_0}\underline{H}_f(\zeta) =f(\zeta_0).
\end{equation*}
for every continuous function $f: \partial \Omega \to \mathbb{R}$.
\end{definition}
Note that we instead could have used $\overline{H}_f$, since $\underline{H}_f =-\overline{H}_{-f}$.
\begin{definition}
	\label{def:barrierdefinition}
A function $w$ is a \emph{barrier} at $\zeta_0 \in \partial \Omega$ if
\begin{enumerate}
\item $w >0$ and $w$ is a supersolution in $\Omega$,
\item $\liminf_{\zeta \to \eta}w(\zeta) >0$ for $\eta\neq\zeta_0 \in \partial \Omega$,
\item $\lim_{\zeta \to \zeta_0}w(\zeta) =0.$
\end{enumerate}
\end{definition}
We shall now prove Theorem \ref{thm:barrier}. We state the result again for convenience.
\begin{theorem*}
A boundary point $\zeta_0$ is regular if and only if there exists a barrier at $\zeta_0.$
\end{theorem*}
\begin{proof}[Proof of Theorem \ref{thm:barrier}]
This is a classical proof. Assume first that there exists a barrier $w$ at $\zeta_0$. Let $\epsilon >0$. By continuity
 $|f(\zeta)-f(\zeta_0)|<\epsilon$ whenever $|\zeta-\zeta_0|<\delta.$  choose a constant $M$ so that
\begin{equation*}
Mw(\zeta_0)\geq 2||f||_{\infty},
\end{equation*}
whenever $|\zeta-\zeta_0|\geq\delta.$ Consequently, the function $Mw+\epsilon+f(\zeta_0)$ is in the upper class $\mathcal{U}_f$, and has the limit $f(\zeta_0)+\epsilon$ as $\zeta \to \zeta_0$, while $-Mw-\epsilon +f(\zeta_0)\in \mathcal{L}_f$ and has the limit $f(\zeta_0)-\epsilon$. Thus
\[
-Mw(\zeta)-\epsilon +f(\zeta_0) \leq \underline{H}_f(\zeta) \leq \overline{H}_f(\zeta) \leq Mw(\zeta)+\epsilon+f(\zeta_0),
\]
or
\begin{align*}
&|\underline{H}_f(\zeta)-f(\zeta_0)|\leq \epsilon + Mw(\zeta), \\ &|\overline{H}_f(\zeta)-f(\zeta_0)|\leq \epsilon + Mw(\zeta).
\end{align*}
Since $w(\zeta) \to 0$ as $\zeta \to \zeta_0$, we see that $\underline{H}_f(\zeta)=\overline{H}_f(\zeta) \to f(\zeta_0)$ as $\zeta \to \zeta_0$.This concludes this direction.

For the other direction, assume that $\zeta_0 =(x_0, t_0)$ is a regular point. We aim at constructing a barrier at $\zeta_0$. To this end, define
\begin{equation*}
\Psi(x, t) =|x-x_0|^2+\epsilon(t-t_0)^2.
\end{equation*}
Assume $x\neq x_0$. Then $D\Psi(x, t)\neq 0$, and it suffices to show that $\Psi$ is a classical supersolution. A quick calculation yields
\begin{equation*}
\Psi_t(x, t)-\Delta^N_{\infty}\Psi(x, t) =2\epsilon(t-t_0)-2
\end{equation*}
so if we choose $\epsilon$ so that
\begin{equation*}
0< \epsilon \cdot \text{diam}(\Omega)<1
\end{equation*}
we get that the above is negative, and so $\Psi$ is a subsolution in $\Omega$.

If $x=x_0$, a similar argument as in the proof of Theorem \ref{thm:exterior sphere} shows that $\Psi$ is a viscosity subsolution even here.

Furthermore, $w=\underline{H}_{\Psi}$ is a barrier at $\zeta_0$ since
\begin{equation*}
\lim_{\zeta \to \zeta_0}w(\zeta) =\Psi(\zeta_0) =0,
\end{equation*} 
by the regularity of $\zeta_0$. This shows that 3) holds. 1) follows from Section \ref{sec:perron=viscosity}, and 2) holds because of the comparison principle \ref{thm:comparison_general}.
\end{proof}
\begin{remark}
The regularity of a boundary point is a very delicate issue. For the heat equation, it can happen that a boundary point is regular for $u_t-\Delta u =0$, while it is irregular for $u_t-\frac{1}{2}\Delta u=0.$ Such a domain can be constructed using Petrovsky's criterion, cf. \cite{watson2012introduction}.  Since \eqref{eq:dirichlet problem} coincides with the heat equation when we consider only one spatial variable,  we get that the same holds true for the $\infty$ -parabolic equation by adding dummy variables.
\end{remark}
For the "other" $\infty$ -heat equation,
\begin{equation}
\label{eq:other}
u_t =\sum_{i, j =1}^{n}u_{x_i}u_{x_ix_j}u_{x_j},
\end{equation}
a similar result holds, but then we need the existence of a \emph{family} of barriers, a single barrier is in general not known to suffice. Indeed, for the related $p$-parabolic equation \eqref{eq:p-parabolic}, it is shown that a single barrier will not suffice if $1<p<2$,  see \cite{Bjorn2015}. 
This difficulty stems from the fact that barriers cannot be multiplied by constants. See the Appendix for more on this.

\hspace{0.5cm}

In the above, we used the existence of a global barrier. In fact, we can do with a local barrier.
\begin{lemma}
	\label{lem:localbarrier}
Let $\zeta_0 \in \partial \Omega$, and let $\tilde{\Omega}$ be a domain such that $\tilde{\Omega}\cap \overline{B}=\Omega \cap \overline{B}$ for some ball $B$ centred at $\zeta_0$. Then there is a barrier in $\Omega$ at $\zeta_0$ precisely when there is a barrier in $\tilde{\Omega}$.
\end{lemma}
\begin{proof}
Let $w$ be a barrier at $\zeta_0$ in $\tilde{\Omega}$. Define
\begin{equation*}
m=\inf\{w(\zeta) : \zeta \in \partial B \cap \tilde{\Omega}\}. 
\end{equation*}
Then the function
\begin{equation*}
v = \begin{cases}
\min(w, m) \ &\text{in} \ B \cap \tilde{\Omega} \\
m \ &\text{in} \ \Omega \setminus B
\end{cases}
\end{equation*}
is easily seen to be a barrier in $\Omega$. Indeed, since $w$ is assumed to be a barrier in $\tilde{\Omega}$, $\left.w\right|_{\tilde{\Omega}}> 0$ by the definition, and therefore $m>0$ and $v>0$. Since $\tilde{\Omega}\cap \overline{B}=\Omega \cap \overline{B}$ we see that $\liminf_{\eta \to \zeta}v(\eta)=\liminf_{\eta \to \zeta}w(\eta)>0$ on $\partial \tilde{\Omega} \cap \overline{B}$ and $\liminf_{\eta \to \zeta}v(\eta)=m>0$ elsewhere. At last, $\lim_{\zeta \to \zeta_0}v(\zeta)=\lim_{\zeta \to \zeta_0}w(\zeta)=0$.  
\end{proof}
From this we get the following useful corollary.
 \begin{corollary}
 	\label{cor:irregular}
 	Let $\tilde{\Omega} \subset \Omega$, and let $\zeta_0$ be a common boundary point. If $\zeta_0$ is \emph{not} a regular point for $\tilde{\Omega}$, then it is not a regular point for $\Omega$.
 \end{corollary}
\begin{proof}
Let $\tilde{\Omega} \subset \Omega$, and let $\zeta_0$ be an irregular boundary point for $\tilde{\Omega}$. Assume that $\zeta_0$ is regular for $\Omega$. Then Theorem \ref{thm:barrier} gives that there exists a barrier, $w$, in $\Omega$. Lemma \ref{lem:localbarrier} implies the existence of a barrier in $\tilde\Omega$, contradicting the irregularity of $\zeta_0$.
\end{proof}
\begin{theorem}[Exterior sphere]
	\label{thm:exterior sphere}
Let $\zeta_0 =(x_0, t_0) \in \partial \Omega$, and suppose that there exists a closed ball
 \[
\{(x, t) \ | \ |x-x'|^2+(t-t')^2 \leq R_0^2\}
\] 
intersecting $\overline{\Omega}$ precisely at $\zeta_0$. Then $\zeta_0$ is regular, if the intersection point is \underline{not} the south pole, that is $(x_0, t_0)  \neq (x', t'-R_0)$.

If the point of contact is the north pole, i.e  $(x, t)=(x', t'+R_0)$, the result is true if the intersecting ball has radius $R_0 \geq 1$.
\end{theorem}
\begin{proof}
We shall construct a suitable barrier at $\zeta_0$. Define
\begin{equation*}
w(x, t) =\mathrm{e}^{-aR_0^2} -\mathrm{e}^{-aR^2},
\end{equation*}
with $R^2 = |x-x'|^2+(t-t')^2 $, and $a>0$ a constant to be chosen later. Clearly, $w(x_0, t_0)=0$, and $w(x, t)>0$ if $\partial\Omega\ni(x, t)\neq (x_0, t_0)$. Close to $\zeta_0$ we have that 
\begin{equation*}
\delta< |x-x'| \ \text{and} \ -2R_0 \leq t-t'. 
\end{equation*}
First we show that $w$ is a viscosity supersolution in $\Omega$.
A calculation yields
\begin{equation}
\label{eq:w-derivatives}
\begin{split}
&Dw(x, t) = 2a\mathrm{e}^{-aR^2}(x-x'),\\
&\ w_t(x, t)=2a\mathrm{e}^{-aR^2}(t-t'), \\
&D^2w(x, t)=2a\mathrm{e}^{-aR^2}(\mathbb{Id}_n-2a(x-x')\otimes(x-x')).
\end{split} 
\end{equation}
{\bf 1.} Assume that $x\neq x'$. Then \eqref{eq:w-derivatives} shows that $Dw(x, t)\neq 0$, and since $w$ is smooth we only need to check that $w$ is a classical supersolution. To this end we see
\begin{align*}
 w_t-\Delta_{\infty}^Nw &=2a\mathrm{e}^{-aR^2}[(t-t')-1+2a|x-x'|^2] \\
 &\geq 2a\mathrm{e}^{-aR^2}[(t-t')-1 +2a\delta^2] \\
 &\geq 2a\mathrm{e}^{-aR^2}[-2R_0-1+2a\delta^2].
\end{align*}
Choosing $a$ so that $2a\delta^2 \geq 2R_0+1$ ensures that the above is positive, and hence $w$ is a supersolution, and  Theorem \ref{thm:barrier} gives that $\zeta_0$ is regular.

{\bf 2.} When the contact point is the north pole, i.e. $(x_0, t_0) =(x', t'+R_0)$, special consideration is needed. It may happen that $x=x'$, and therefore $Dw=0$ in a neighbourhood of $(x_0, t_0)$. For the points where $x\neq x'$, we calculate
\begin{align*}
 w_t-\Delta_{\infty}^Nw &=2a\mathrm{e}^{-aR^2}[(t-t')-1+2a|x-x'|^2] \\
 &\geq 2a\mathrm{e}^{-aR^2}[(t-t')-1]. 
 \end{align*}
For this to be positive, we must demand that $(t-t')>1$ at the intersection point, and this implies that the sphere must have radius $R_0$ greater than 1. If this restriction on the radius can be circumvented is unknown.

On the other hand, if $x=x'$ we need to show that for every $\phi\in C^2(\Omega)$ touching $w$ from below at $(x', t)$ we have
\begin{equation}
\label{eq:w-visc-super}
\phi_t(x', t) \geq \lambda(D^2\phi(x', t)).
\end{equation}
We make the following \emph{antithesis}: Assume there is a $\phi$ that touches $w$ from below at $(x', t)$ but
\begin{equation}
\label{eq:w-antithesis}
\phi_t(x', t) < \lambda(D^2\phi(x', t)).
\end{equation}
Since $\phi$ is touching from below, we get
\begin{align*}
\phi_t(x', t)=u_t(x', t), \ D\phi(x', t)=Du(x', t), \ D^2u(x', t) \geq D^2\phi(x', t).
\end{align*}
This implies, for any $z\in \mathbb{R}^n$
\begin{align*}
\langle D^2wz, z\rangle \geq \langle D^2\phi z, z\rangle \geq \lambda(D^2\phi)|z|^2> \phi_t|z|^2
\end{align*}
at $(x', t)$. Inserting $D^2w(x', t)=2a\mathrm{e}^{-aR^2}\mathbb{Id}_n$, we see
\[
2a\mathrm{e}^{-aR^2}|z|^2>\phi_t(x', t)|z|^2,
\]
and dividing the inequality by $|z|^2$ and inserting $\phi_t(x',t)=w_t(x', t)=2a\mathrm{e}^{-aR^2}(t-t')$, we get
\[
2a\mathrm{e}^{-aR^2}>2a\mathrm{e}^{-aR^2}(t-t'),
\]
a contradiction, since we demanded that $t-t'>1$. Hence \eqref{eq:w-visc-super} holds, and $w$ is a viscosity supersolution.
\end{proof}
We stress that it was strictly necessary to exclude the south pole. To see why this is true, let $\zeta_0 =(x_0, t_0)  = (x', t'-R_0).$ Close to $\zeta_0$ we have $0 <|x-x'| <\sigma$ and $t-(t'-R_0) <0. $ Now,
\begin{align*}
w_t-\Delta_{\infty}^Nw &=2a\mathrm{e}^{-aR^2}[(t-t')-1+2a|x-x'|^2] \\
&<2a\mathrm{e}^{-aR^2}[-R_0-1+2|x-x'|^2] \\
&<2a\mathrm{e}^{-aR^2}[-R_0-1+2a\sigma^2] <0
\end{align*}
for $\sigma$ small enough.

Another way to see this, is to consider the Dirichlet problem on the cylinder $\Omega = Q_T =Q\times (0, T)$, and suppose that $f : \partial \Omega \to \mathbb{R}$ is continuous. Theorem \ref{thm:comparison} and \ref{thm:existence} gives the existence of a unique viscosity solution $h$ in $\Omega$.

Now construct the upper and lower Perron solutions $\overline{H}_f$ and $\underline{H}_f$. Since both are solutions in $\Omega$, uniqueness gives that $\underline{H}_f = \overline{H}_f=h$, regardless of what values we choose at that part of the boundary where $t=T$. Indeed, $h$ itself need not be in either the upper or lower class, because we may not have either $h>f$ or $h<f$ on the plane $t=T$. However, if we define
\begin{equation*}
\tilde{h} =h(x, t)+\frac{\epsilon}{T-t},
\end{equation*}
we see that
\begin{equation*}
\tilde{h}_t-\Delta_{\infty}^N\tilde{h} =0+\frac{\epsilon}{(T-t)^2},
\end{equation*}
so $\tilde{h}$ is in $\mathcal{U}_f$ for $\epsilon >0$, and in $\mathcal{L}_f$ for $\epsilon <0$. 

%
%
\begin{example}[$\infty$- Heat Balls]
	\label{ex:balls}
	Juutinen \& Kawohl in \cite{Juutinen2006}  describe the solution
	\begin{equation}
	\label{eq:fundamental}
	W(x, t) = \frac{1}{\sqrt{t}}\mathrm{e}^{-\frac{|x|^2}{4t}}
	\end{equation}
	to the $\infty$-heat equation, and compare it to the fundamental solution 
	\begin{equation*}
	H(x, t) = \frac{1}{(4\pi t)^{n/2}}\mathrm{e}^{-\frac{|x|^2}{4t}}.
	\end{equation*}
	to the heat equation,
	Since the classical heat balls are defined by the level sets
	\begin{equation*}
	H(x_0-x, t_0-t)> c,
	\end{equation*}
	one can surmise that the $\infty$- heat balls are defined by
	\begin{equation}
	\label{eq:heat balls}
	W(x_0-x, t_0-t)> c.
	\end{equation}
	The centre $(x_0, t_0)$ is an irregular boundary point of the heat ball, see for example \cite{watson2012introduction}. The Petrovsky criterion shows that the same is true for the $\infty$ -parabolic balls.
	
	We can assume that $(x_0, t_0) = (0, 0)$. Equation \eqref{eq:heat balls} gives the inequality
	\begin{equation*}
	\mathrm{e}^{-\frac{|x|^2}{4(-t)}} >c\sqrt{-t},
	\end{equation*}
	which is equivalent to
	\[
	|x|^2<t(\log{|t|}+\log{c})
	\]
	This domain encircles the domain described in the Petrovsky criterion, Theorem \ref{thm:not regular}, and hence the origin is not a regular point.
\end{example}
\subsection{The Future cannot influence the present}
Fix a boundary point $(x_0, t_0)$, and define 
\begin{align*}
&\Omega_- =\{(x, t) \in \Omega : t_0>t\}, \\ 
&\Omega_+ =\{(x, t) \in \Omega : t_0<t\}.
\end{align*}

\begin{theorem}
\label{thm:future}
Let $\zeta_0 = (x_0, t_0) \in \partial \Omega$. Then $\zeta_0$ is a regular boundary point of $\Omega$ if and only if $\zeta_0$ is a regular boundary point of $\Omega_-$, or $\zeta_0 \notin \partial \Omega_-$.
\end{theorem}
\begin{proof}
The first direction is easy. Assume that $\zeta_0$ is a regular boundary point for $\Omega$. Then there exists a barrier, $\psi$, at $\zeta_0$. Since $\Omega_- \subset \Omega$, $\psi$ is also a barrier in $\Omega_-$, or $\zeta_0\notin \partial\Omega_-$. In particular, condition (2) in Definition \ref{def:barrierdefinition} is satisfied since $\psi>0$ in all of $\Omega$.

For the converse, assume first that $\zeta_0\notin \Omega_-$. Then the exterior sphere condition guarantees that $\zeta_0$ is a regular boundary point for $\Omega$. 

Assume that $\zeta_0 \in \partial \Omega_-$ is a regular boundary point. We show that this implies  $\zeta_0$ is a regular boundary point for all of $\Omega$, by constructing a barrier at $\zeta_0$ for $\Omega.$

Define the boundary values $\Psi: \partial \Omega \to \mathbb{R}$ by
\begin{equation*}
\Psi(x, t) =|x-x_0|^2+\epsilon(t-t_0)^2,
\end{equation*}
where $0 < \epsilon <\frac{1}{\text{diam}(\Omega)}$. Then $\Psi$ is a subsolution, and the Lower Perron solution $\underline{H}_\Psi$ satisfies
\begin{equation*}
\underline{H}_\Psi \geq \Psi \geq 0
\end{equation*}
in $\Omega$. We want to show that 
\begin{equation}
\label{eq:goodlimit}
\lim_{\Omega_- \ni\zeta \to \zeta_0}\underline{H}_\Psi(\zeta)=\Psi(\zeta_0)=0.
\end{equation}
To see this, pick any function $u$ in the upper class for $\Psi$ in $\Omega_-$ (That is, $u$ is a supersolution in $\Omega_-$, and $\liminf_{\zeta \to \nu} u(\zeta)\geq \Psi(\nu)$ for all $\nu\in \partial \Omega_-$). Then, for each $\delta >0$, we have that the function
\begin{equation*}
v= \begin{cases}
\sup_{\zeta \in \Omega_-}u(\zeta), \ &\text{for} \ t>t_0-\delta \\
u, \ &\text{for} \ t\leq t_0-\delta
\end{cases}
\end{equation*}
is in the upper class for $\Psi$ in all of $\Omega.$ Indeed, clearly $\limsup_{\zeta \to \nu}v(\zeta)\geq \Psi(\nu)$ for any $\nu \in \partial \Omega$. 

Furthermore, if $t \leq t_0-\delta$, we have that whenever $\theta \in C^2(\Omega)$ touches $v$ from below at a point $\zeta$, we also have that $\theta $ touches $u$ from below. Hence $\theta_t(\zeta)-\Delta^N_{\infty}\theta(\zeta) \geq 0$. For $t>t_0-\delta$, $v$ is a constant and hence a classical (super)solution, and hence a viscosity supersolution.

It now follows that $\underline{H}_\Psi$ restricted to $\Omega_-$ is precisely the upper Perron solution of $\Psi$ in $\Omega_-$, and hence \eqref{eq:goodlimit} holds, since we assumed that $\zeta_0$ is a regular boundary point of $\Omega_-$.

Since $H_{\Psi}$ is positive and a supersolution in $\Omega_-$, we see that $H_{\Psi}$ is a barrier for $\zeta_0$ in $\Omega_-$. We claim that $H_{\Psi}$ will do as a barrier in all of $\Omega$ as well. It only remains to prove that
\begin{equation}
\label{eq:goodlimit2}
\lim_{\Omega \ni \zeta \to \zeta_0}H_{\Psi}(\zeta) =0
\end{equation}
Let $\zeta_0\in \partial \Omega_+$. Define the function
\begin{equation*}
\phi= 
\begin{cases}
\Psi \ \text{on} \ \partial \Omega \\
H_{\Psi} \ \text{on} \ \Omega \cap \partial \Omega_+.
\end{cases}
\end{equation*}
Then $\phi$ restricted to $\partial \Omega_+$ is continuous at $\zeta_0$. Now let 
\begin{equation*}
h=\underline{H}_{\phi}
\end{equation*}
be the lower Perron solution of $\phi$ in $\Omega_+$. We claim that $h=H_{\Psi}$. To see this, pick a $u\in \mathcal{L}_\Psi(\Omega)$. Then $u$ restricted to $\Omega_+$ is in the Lower Class $\mathcal{L}_{\phi}(\Omega_+)$, and so
\begin{equation*}
\left.u \right |_{\Omega_+} \leq h,
\end{equation*}
and hence $H_{\Psi} \leq h$ in $\Omega_+$. To prove that $H_{\Psi} \geq h$, pick a $v\in \mathcal{L}_{\phi}(\Omega_+)$. Then we have that 
\begin{equation*}
\limsup_{\eta \to \zeta}v(\eta) \leq \phi(\zeta) \leq \liminf_{\eta \to \zeta}H_{\Psi}
\end{equation*}
for $\zeta \in \partial \Omega_+$, and so the comparison principle gives that $v\leq H_{\Psi}$ in $\Omega_+$.Hence $h \leq H_{\Psi}$, and we have proven that $h=H_{\Psi}$ in $\Omega_+$.

The exterior sphere condition gives that the earliest points of $\Omega$ are regular, $\zeta_0$ is a regular boundary point of $\Omega_+$. This implies that
\begin{equation*}
\lim_{\zeta \to \zeta_0}h(\zeta) =\phi(\zeta_0) =0.
\end{equation*}
Since $H_{\Psi}=h$ in $\Omega_+$, this and \eqref{eq:goodlimit} implies \eqref{eq:goodlimit2}.
\end{proof}
\section{Regular points on Cylinders}
\label{sec:regularcylinders}
\noindent We provide the proof of Theorem \ref{thm:regularcylinders}.
The parabolic boundary of the cylinder $Q_T=Q\times (0, T)$ is divided into three parts:
\begin{enumerate}
\item The ``bottom'' $Q\times\{t=0\}$
\item The boundary points of the bottom $\partial Q \times\{t=0\}$
\item The walls $\partial Q \times (0, T]$
\end{enumerate}
For the bottom, we have that if $(x_0, 0)\in Q \times\{t=0\}$, then the function
\begin{equation*}
\Psi(x, t) = |x-x_0|^2+2t
\end{equation*}
is positive in $\overline{Q_T}$ and equal to zero at $(x_0, 0)$. Arguing as in Theorem \ref{thm:exterior sphere}, we can show that $\Psi$ is a viscosity solution in $Q_T$ and hence a barrier. This also shows that points on the boundary of the bottom are regular, since $\Psi$ is a barrier here, as well.


For the walls, we have the following characterization, which is Theorem \ref{thm:regularcylinders}.
\begin{theorem*}
Let $x_0 \in \partial Q$ and $0<t_0<T$. Then $\zeta_0 =(x_0, t_0)$ is a  regular boundary point of $Q_T$.
\end{theorem*}
\begin{proof}
Since $\infty$-harmonic functions are Lipschitz continuous in $\overline{Q}$, cf. \cite{aronsson}, every boundary point is regular for the $\infty$-Laplace equation \eqref{eq:stationaryeg} To see why, let $p>n$ and $f\in C(\partial Q)$. Fix a point $x_0 \in \partial Q$. By continuity, we can find an $\epsilon >0$ so that $|f(x)-f(x_0)|<\epsilon$ whenever $|x-x_0|<\delta$ and $x \in \partial Q$. Let $u_p$ be the solution of the $p$-Laplace equation
\begin{equation}
\label{eq:p-laplace}
\begin{cases}
\text{div}(Du|Du|^{p-2})=0 \ \text{in} \ Q \\
u_p =f \ \text{on} \ \partial Q
\end{cases}
\end{equation}
 Classical results, see e.g  \cite{Lind_rev}, give that $u_p$ is continuous in $\overline{Q}$.
 Let $M=\max_{x \in \partial Q}{f(x)}$.
 Since $u_p$ enjoys a comparison principle, we get the following inequality in $Q$ where $\lambda = 2M/\delta^{(p-n/(p-1))}$:
 \begin{equation}
 \label{eq:pcomparison}
 f(x_0)-\epsilon -\lambda|x-x_0|^{\frac{p-n}{p-1}} \leq u_p(x)\leq  f(x_0)+\epsilon +\lambda|x-x_0|^{\frac{p-n}{p-1}},
 \end{equation}
 since $f(x_0)-\epsilon -\lambda|x-x_0|^{\frac{p-n}{p-1}}$ and $f(x_0)+\epsilon +\lambda|x-x_0|^{\frac{p-n}{p-1}}$ are respectively sub and supersolutions of \eqref{eq:p-laplace} and since \eqref{eq:pcomparison} holds on $\partial Q$. As $p \to \infty$, the $u_p$ converges to the $\infty-$harmonic function $u_{\infty}$ pointwise in $Q$ and so 
 \[
 f(x_0)-\epsilon -\frac{2M}{\delta}|x-x_0| \leq u_{\infty}(x)\leq  f(x_0)+\epsilon +\frac{2M}{\delta}|x-x_0|.
 \]
 As $x \to x_0$, we get
 \[
 f(x_0)-\epsilon \leq \liminf_{x \to x_0}u_{\infty}(x)\leq \limsup_{x \to x_0}u_{\infty}(x)\leq f(x_0)+\epsilon, 
 \]
 and since $\epsilon$ was arbitrary,
 \[
 \lim_{x \to x_0}u_{\infty}(x)=f(x_0).
  \]
  
 Let $x_0\in \partial Q$. We shall prove that $(x_0, t_0)$ is a regular point for \eqref{eq:dirichlet problem} in $Q_T$.
  Since $x_0$ is regular for the $\infty$-Laplacian, we have for the viscosity solution of the $\infty$-Laplace equation $u_{\infty}$  that
  \[
   \lim_{Q\ni x \to x_0}u_{\infty}(x)=\phi(x_0)
  \]
for every continuous function $\phi$ defined on the boundary $\partial Q$. We shall prove that $(x_0, t_0)$ is regular by constructing a suitable barrier in $Q_T$

 Let $\phi(x)=|x-x_0|$, and let $\nu$ be a solution of 
\begin{equation*}
\begin{cases}
\Delta_{\infty}^N\nu =-1 \ \text{in} \ Q \\
\nu -\phi =0 \ \text{on} \ \partial Q
\end{cases}
\end{equation*}
Then $\nu$ is $\infty$-superharmonic in $Q$, and since $\phi$ is $\infty$-subharmonic, comparison \cite{lu2008inhomogeneous} gives that $\nu(x)\geq |x-x_0|$. We also have that
\begin{equation*}
\lim_{x \to x_0}\nu(x) =\phi(x_0)=0,
\end{equation*}
and hence $\nu$ is a barrier in $Q$ for the \emph{$\infty$-Laplacian}.
Define 
\begin{equation*}
v(x, t)=\nu(x)+(t_0-t)
\end{equation*}
so that 
\begin{equation*}
v_t -\Delta_{\infty}^Nv =0.
\end{equation*}
This implies that $v$ is a barrier at $(x_0, t_0)$ in $Q\times(0, t_0)$, and Theorem \ref{thm:future} gives that $(x_0, t_0)$ is regular in $\Omega.$
%
%
Thus every point on the parabolic boundary of the cylinder $Q_T$ is regular, and Theorem \ref{thm:regularcylinders} is proven.
\end{proof}

\section{The Petrovsky Criterion}
\label{sec:petrowski}
\noindent Here we prove Theorem \ref{thm:petrowski}, repeated for completeness.
\begin{theorem*}
	The origin $(x, t)=(0, 0)$ is a regular point for 
	\begin{equation}
	\label{eq:1}
	u_t - \Delta_{\infty}^N u=0
	\end{equation}
	in the domain enclosed by the hypersurfaces 
	\begin{equation}
	\label{eq:thedomain}
	\{(x, t)\in \mathbb{R}^n\times (-\infty, 0)\ : \ |x|^2 =-4t\log|\log|t||\} \ \text{and} \ \{t=-c\}, 
	\end{equation}
	for a small constant $0<c<1$.
\end{theorem*}
\begin{proof}[Proof of Theorem \ref{thm:petrowski}]
	We shall construct a barrier function $\overline{u}$ so that
	\begin{enumerate}
		\item $\overline{u} >0$ and $\overline{u}$ is a supersolution in $\Omega$, that is \\ $\overline{u}_t - \Delta_{\infty}^N \overline{u}\geq0$.
		\item $\liminf_{(y, s) \to (x, t)}\overline{u}(y, s) >0$ for $(x, t) \neq (0, 0) \in \partial \Omega$,
		\item $\lim_{(x, t) \to (0, 0)}\overline{u}(x, t) =0.$
	\end{enumerate}
	We will choose a function on the form
	\begin{equation}
	\overline{u}(x, t)=f(t)\mathrm{e}^{\frac{-|x|^2}{4t}} +g(t).
	\end{equation}
	Inserting this into \eqref{eq:1} we get
	\begin{equation}
	\label{eq:mustbepositive}
	\overline{u}_t(x, t)-\Delta_{\infty}^N\overline{u}(x, t) = \mathrm{e}^{\frac{-|x|^2}{4t}}\left(f'(t)+\frac{f(t)}{2t}+g'(t)\mathrm{e}^{\frac{|x|^2}{4t}}\right).
	\end{equation}
	Choose
	\[
	f(t)=-\frac{1}{2}\frac{1}{|\log|t||^{\delta+1}}, \ \ g(t)=\frac{1}{|\log|t||^{\delta}}
	\]
	where $0<\delta \leq 1/4$ is a parameter to be fixed later.
	
	{\bf 1.} 	
We first show that
	\begin{equation}
	\label{eq:barrier}
	\overline{u}(x, t) =-\frac{1}{2}\frac{1}{|\log|t||^{\delta+1}}\mathrm{e}^{\frac{-|x|^2}{4t}}+\frac{1}{|\log|t||^{\delta}}
	\end{equation}
	is positive in the domain \eqref{eq:thedomain}. The definition of the domain implies
	\[
	-\frac{|x|^2}{4t}<\log|\log|t||,
	\]
	and so 
	\[
	\overline{u}(x, t)>-\frac{1}{2|\log|t||^{\delta+1}}\mathrm{e}^{\log|\log|t||} +\frac{1}{|\log|t||^{\delta}} =\frac{1}{2|\log|t||^{\delta}}>0.
	\]
	
	 Then we show that $\overline{u}$ is a viscosity supersolution in the domain. A calculation yields
	\begin{equation}
	\label{eq:gradient}
	Du(x, t)=-x\frac{f(t)}{2t}\mathrm{e}^{\frac{-|x|^2}{4t}},
	\end{equation}
	\begin{equation}
	\label{eq:timederivative}
	u_t(x, t)=\mathrm{e}^{\frac{-|x|^2}{4t}}\left(f'(t)+\frac{|x|^2}{4t^2}\right)+g'(t),
	\end{equation}
	and
	\begin{equation}
	\label{eq:hessian}
	D^2u(x, t)=-\frac{f(t)}{2t}\mathrm{e}^{\frac{-|x|^2}{4t}}\left(\mathbb{Id}_n-\frac{1}{2t}x\otimes x\right),
	\end{equation}
	where $\mathbb{Id}_n$ is the $n\times n$ identity matrix and $x \otimes x$ the tensor product. We see that $Du(x, t)=0$ only when $x=0$. This case needs to be checked separately.
	
	Assume first that $x\neq 0$. Then $Du(x, t)\neq 0$, and, since $u$ is smooth it suffices to show that $u$ is a classical supersolution.
	Differentiating, we see
	\[
	f'(t)=-\frac{\delta+1}{2}\frac{1}{t|\log|t||^{\delta +2}}, \ \ g'(t)=\delta\frac{1}{t|\log|t||^{\delta+1}}.
	\]
	Inserting  the derivatives into \eqref{eq:mustbepositive} we get
	\begin{align*}
	&\overline{u}_t(x, t)-\Delta_{\infty}^N\overline{u}(x, t)  \\
	=&\mathrm{e}^{\frac{-|x|^2}{4t}}\left(-\frac{\delta +1}{2}\frac{1}{t|\log|t||^{\delta +2}}-\frac{1}{2}\frac{1}{2t|\log|t||^{\delta+1}}+\delta\frac{1}{t|\log|t||^{\delta+1}}\mathrm{e}^{\frac{|x|^2}{4t}}\right) \\
	=&\mathrm{e}^{\frac{-|x|^2}{4t}}\frac{1}{t|\log|t||^{\delta+1}}\left(\frac{-(\delta+1)}{2|\log|t||}-\frac{1}{4}+\delta\mathrm{e}^{\frac{|x|^2}{4t}}\right).
	\end{align*}
	Since $t<0$, we have
	\[
	\mathrm{e}^{\frac{|x|^2}{4t}}<1.
	\]
	Hence
	\[
	\overline{u}_t(x, t)-\Delta_{\infty}^N\overline{u}(x, t) >\mathrm{e}^{\frac{-|x|^2}{4t}}\frac{1}{t|\log|t||^{\delta+1}}\left(\frac{-(\delta+1)}{2|\log|t||}-\frac{1}{4}+\delta\right).
	\]
	Setting $\delta =1/4$ ensures that this is positive.
	
	Then assume $x=0$. We need to check that for every $\phi \in C^2(\Omega)$ so that $u-\phi$ has a minimum at $(0, t)$, we have
	\begin{equation}
	\label{eq:viscdef}
	\phi_t(0, t) \geq \lambda(D^2\phi(0, t)).
	\end{equation}
	We first show that $\overline{u}$ itself satisfies \eqref{eq:viscdef}. Inserting $x=0$ into \eqref{eq:hessian} and \eqref{eq:timederivative}, we see that we must have
	\[
	f'(t)+g'(t) \geq \lambda(-\frac{f(t)}{2t}\mathbb{Id}_n) =-\frac{f(t)}{2t} = \frac{1}{4t|\log|t||^{\delta+1}}.
	\]
	Writing out the left hand side, we get 
	\[
	\frac{1}{t|\log|t||^{\delta+1}}\left(-\frac{\delta+1}{2}\frac{1}{|\log|t||}+\delta\right) \geq  \frac{1}{4t|\log|t||^{\delta+1}},
	\]
	which is equivalent to
	\[
	-\frac{\delta+1}{2}\frac{1}{|\log|t||}+\delta \leq \frac{1}{4}.
	\]
	Since $\delta$ was set to  $1/4$, we end up with
	\[
	-\frac{5}{8}\frac{1}{|\log|t||} \leq 0,
	\]
	 and so $\overline{u}$ itself satisfies \eqref{eq:viscdef} in $\Omega$. Let $\phi$ be the test function in Definition \ref{def:viscositydefinition} so that $\overline{u}-\phi$ has a minimum at $(0, t)$. This implies that 
	\[
	D\overline{u}=D\phi, \ \overline{u}_t = \phi_t, \ D^2(\overline{u}-\phi) \geq 0
	\]
	at $(0, t)$. Since $D^2\overline{u}(0, t)$ is a scalar multiple of the identity matrix, $D^2\overline{u} \geq D^2\phi$ implies that $\lambda(D^2\overline{u}(0, t))\geq \lambda(D^2\phi(0, t))$. Hence
	\[
	\phi_t(0, t)=\overline{u}_t(0, t) \geq  \lambda(D^2\overline{u}(0, t)) \geq \lambda(D^2\phi(0, t)) ,
	\]
	and  so \eqref{eq:viscdef} holds. Thus $\overline{u}$ is a viscosity supersolution even in this case.
	
	{\bf 2} We show that $\overline{u}$ satisfies condition (2). $\overline{u}$ is continuous up to $\partial \Omega$ and, since $|x|^2 =-4t\log|\log|t||$ here we have
	\begin{align*}
	&\left.\overline{u}(x, t)\right|_{(|x|^2=-4t\log|\log|t||)}  \\
	=&\left(-\frac{1}{2}\frac{1}{|\log|t||^{\delta+1}}\mathrm{e}^{-\frac{-4t\log|\log|t||}{4t}}+\frac{1}{|\log|t||^{\delta}}\right) \\
	=&\frac{-1}{2|\log|t||^{\delta}} +\frac{1}{|\log|t||^{\delta}} =\frac{1}{2|\log|t||^{\delta}}>0,
	\end{align*}
	for $(x, t)\neq (0, 0)$.
	
	{\bf 3.} Since $f(t), g(t) \to 0$ as $t\to 0^-$, we see that (3) is satisfied. Indeed, since $|x|^2 <-4t\log|\log|t||\to 0$ as $t \to 0^-$, we get
	\[
	\mathrm{e}^{\frac{-|x|^2}{4t}} =\mathcal{O}(|\log|t||) 
	\]
	as $t\to 0^-$. Hence, taking the limit from within $\Omega$,
	\begin{align*}
	&\lim_{(|x|, t) \to (0,0^-)}\overline{u}(x, t)= \\
	  &\lim_{(|x|, t) \to (0,0^-)}\left(-\frac{1}{2}\frac{1}{|\log|t||^{\delta+1}}\mathrm{e}^{\frac{-|x|^2}{4t}}+\frac{1}{|\log|t||^{\delta}}\right)=0.
	\end{align*} 

	Hence \eqref{eq:barrier} satisfies conditions (1), (2) and (3), and therefore is a barrier at $(0, 0)$. This shows that the origin is a regular point for the domain \eqref{eq:thedomain}.
\end{proof}
	\begin{remark}
	One could also deduce $\overline{u}>0$ and condition (2) from the following considerations. Consider the level set $\overline{u}(x, t)=0$. We get
	\begin{align*}
	&-\frac{1}{2}\frac{1}{|\log|t||^{\delta+1}}\mathrm{e}^{\frac{-|x|^2}{4t}}+\frac{1}{|\log|t||^{\delta}}=0, \\
	&-\frac{1}{2}\frac{1}{|\log|t||}\mathrm{e}^{\frac{-|x|^2}{4t}} +1 =0, \\
	&\mathrm{e}^{\frac{-|x|^2}{4t}} =2|\log|t||,
	\end{align*}
	and upon taking logarithms and rearranging;
	\[
	|x|^2=-4t(\log|\log|t||+\log2).
	\]
	This curve includes the domain \eqref{eq:thedomain}.
\end{remark}
 We  now turn to the proof that the constant 4 in the Petrovsky criterion cannot be improved upon. This is Theorem \ref{thm:not regular}.
 \begin{theorem*}
 The origin is \emph{not} a regular point for the domain $\Omega$ bounded by $|x|^2 =-4(1+\epsilon)t\log{|\log{|t|}|}$, $t=-c$ for any $\epsilon >0$.
 \end{theorem*}
 \begin{proof}[Proof of Theorem \ref{thm:not regular}]
 The proof proceeds by constructing a domain $\tilde{\Omega}$ contained in $\Omega$, with the origin as common boundary point. We then show that $(0, 0)$ is irregular for $\tilde{\Omega}$, and Lemma \ref{cor:irregular} then implies that $(0, 0)$ regarded as a boundary point of $\Omega$ is irregular, too. 
 	
 We shall construct a smooth function $v$ so that
 \begin{enumerate}
 \item $ v_t(x, t)-\Delta_{\infty}^Nv(x, t)\leq0$ in $\tilde\Omega$,
 \item $v$ is continuous on the closure  of  $\tilde\Omega$ but not at the origin.
 \item The upper limit of $v$ at interior points converging to $(0, 0)$ is greater than its upper limit for the points converging to $(0,0)$ along the boundary of $\tilde \Omega$.
 \end{enumerate}
The existence of such a $v$ implies that the origin is irregular. Indeed, consider the boundary data $f: \partial \tilde\Omega \to \mathbb{R}$ defined as follows. Let $f=v$ near $(0, 0)$, and define 
\[
f(0, 0)=\lim_{\partial \tilde\Omega\ni(x, t)\to(0, 0)}v(x, t).
\]
As we shall see, this limit along the boundary exists.
 For the rest of the boundary, continuously extend $f$ to a large constant $c$. 

If $c$ is large enough, then the comparison principle implies that every function $w\in \mathcal{U}_f$ which satisfies $w\geq f$ on $\partial \tilde\Omega$ also satisfies $w\geq v$ in $\tilde\Omega$ since $v$ is a subsolution by (1). Taking the infimum over all such $w$, we see $\overline{H}_f\geq v$ in $\tilde\Omega$, and hence by point (3) in the definition of $v$;
\begin{align*}
\limsup_{\tilde\Omega\ni(x, t) \to (0, 0)}\overline{H}_f(x, t) \geq &\limsup_{\tilde\Omega\ni(x, t) \to (0, 0)} v(x, t) \\ 
>&\limsup_{\partial \tilde\Omega\ni(y, s) \to (0, 0)}v(y, s) =f(0, 0),
\end{align*} 
and so $(0, 0)$ is not a regular point for $\tilde \Omega$.

 We shall seek a function on the form
 \begin{equation}
 \label{eq:irregularitatsbarriare}
 v(x, t)=f(t)\mathrm{e}^{\frac{-|x|^2}{4t}k} +g(t),
 \end{equation}
 for suitable functions $f$ and $g$. Here $k$ and $\alpha$ are positive constants to be determined later, and $-1<t<0$. In fact, we shall choose $t$ to be very close to 0 and $1/2<k<1$. 
 
 {\bf The case $\bf{x\neq 0}$.} Then $Dv(x, t)\neq 0$, and it suffices to show that $v$ is a classical subsolution.
 Calculating, we get
 \begin{equation*}
v_t(x, t ) = f'(t)\mathrm{e}^{\frac{-|x|^2}{4t}k}+\frac{f(t)|x|^2k}{4t^2}\mathrm{e}^{\frac{-|x|^2}{4t}k}+g'(t),
 \end{equation*}
 and
 \begin{equation*}
 \Delta_{\infty}^Nv(x, t) = \frac{f(t)|x|^2k^2}{4t^2}\mathrm{e}^{\frac{-|x|^2}{4t}k}-\frac{f(t)k}{2t}\mathrm{e}^{\frac{-|x|^2}{4t}k}.
 \end{equation*}
 Collecting terms, this gives
 \begin{equation}
 \begin{aligned}
  \label{eq:bigstuff}
  & v_t(x, t)-\Delta_{\infty}^Nv(x, t) =   \\ 
  &\mathrm{e}^{\frac{-|x|^2}{4t}k} \left(\frac{f(t)|x|^2(k-k^2)}{4t^2}+f'(t)+\frac{f(t)k}{2t}+g'(t)\mathrm{e}^{\frac{|x|^2}{4t}k}\right).
 \end{aligned}
 \end{equation}
Choose
 \begin{equation*}
 f(t)=\frac{-1}{{|\log{|t|}|}^{1+\alpha}} \ \text{and} \ g(t)= \frac{1}{\log|\log{|t|}|},  
 \end{equation*}
 where $\alpha$ is an arbitrary, positive constant. We differentiate
 \begin{equation*}
 f'(t)=\frac{-(1+\alpha)}{t\cdot|\log{|t|}|^{2+\alpha}}
 \end{equation*}
 and
 \begin{equation*}
 g'(t)= \frac{1}{\log^2|\log{|t|}|\cdot|\log{|t|}|\cdot t}.
 \end{equation*}
 Inserting all this into \eqref{eq:bigstuff}, we end up with
\begin{equation}
\begin{split}
v_t(x, t)-\Delta_{\infty}^Nv(x, t) &= \mathrm{e}^{\frac{-|x|^2}{4t}k}\left(\frac{-|x|^2(k-k^2)}{4t^2|\log{|t|}|^{1+\alpha}}-\frac{1+\alpha}{t\cdot|\log{|t|}|^{2+\alpha}}\right.  \\ -\frac{k}{2t|\log{|t|}|^{1+\alpha}}  &+\left.\mathrm{e}^{\frac{|x|^2}{4t}k}\frac{1}{\log^2|\log{|t|}|\cdot|\log{|t|}|\cdot t}\right).
\end{split}
\end{equation}
 For this to be negative, the expression inside the parentheses has to be negative. Multiplying this by $-t\cdot|\log{|t|}|^{1+\alpha}>0$, we must have
 \begin{equation}
 \label{eq:subparabolicineq}
 \frac{|x|^2(k-k^2)}{4t}+\frac{1+\alpha}{|\log{|t|}|}+\frac{k}{2}- \mathrm{e}^{\frac{|x|^2}{4t}k}\frac{|\log{|t|}|^{\alpha}}{\log^2|\log{|t|}|}<0.
 \end{equation}
 Here, we can choose $t$ so close to 0 that $\log{|t|}<0$ and 
 \begin{equation*}
 \frac{\alpha+1}{|\log{|t|}|} <\frac{k}{2}, 
 \end{equation*}
 and it is enough that 
  \begin{equation}
  \label{eq:negative}
  \frac{|x|^2(k-k^2)}{4t}+k- \mathrm{e}^{\frac{|x|^2}{4t}k}\frac{|\log{|t|}|^{\alpha}}{\log^2|\log{|t|}|}<0.
  \end{equation}
  The expression \eqref{eq:negative} is negative if 
 \begin{equation}
 \label{eq:large}
\frac{|x|^2(k-k^2)}{4|t|} >k,
 \end{equation}
 or if
 \begin{equation}
 \label{eq:small}
 k\mathrm{e}^{-\frac{|x|^2}{4t}k} <\frac{|\log{|t|}|^{\alpha}}{\log^2|\log{|t|}|}.
 \end{equation}
It turns out that at least one of these inequalities must be true, if we choose $\delta_0$ so small that
 \begin{equation}
 \label{eq:choosing delta}
 k\mathrm{e}^{\frac{1}{1-k}} \leq \frac{|\log{|t|}|^{\alpha}}{\log^2|\log{|t|}|}, \ \mbox{when} \ -\delta_0<t<0.
 \end{equation}
 Indeed, if \eqref{eq:large} holds, we are done. 

 On the other hand, if \eqref{eq:large} does \emph{not} hold, we have
 \[
  -\frac{|x|^2}{4t}<\frac{1}{1-k}.
 \]
This implies
\[
k\mathrm{e}^{-\frac{|x|^2}{4t}k}<k\mathrm{e}^{\frac{k}{1-k}},
\]
and so \eqref{eq:small} follows from \eqref{eq:choosing delta}.
At least one of the inequalities \eqref{eq:large} or \eqref{eq:small} is satisfied. This concludes the verification when $x \neq 0$.

{\bf The case $\bf{x= 0}$.} According to Definition \ref{def:viscositydefinition} we need to check that for every $\phi \in C^2(\Omega)$ such that $v-\phi$ has a maximum at $(0, t)$, we have
\begin{equation}
\label{eq:subparabolic}
\phi_t\leq \Lambda (D^2\phi(0, t)).
\end{equation}
We show that $v$ itself satisfies this condition. An argument similar to the one in the proof of Theorem \ref{thm:petrowski} shows that $v$ is then a viscosity subsolution. 

Substituting the derivatives, we se that at $(0, t)$ equation \eqref{eq:subparabolic} reads
\[
f'(t)+g(t)\leq -\frac{f(t)}{2t}k,
\]
or
\[
-\frac{1+\alpha}{t\cdot|\log{|t|}|^{2+\alpha}}  +\frac{1}{\log^2|\log{|t|}|\cdot|\log{|t|}|\cdot t} \leq \frac{k}{2t|\log{|t|}|^{1+\alpha}}. 
\]
This is equivalent to 
\[
\frac{1+\alpha}{|\log{|t|}|}-\frac{|\log{|t|}|^{\alpha}}{\log^2|\log{|t|}|}\leq-\frac{k}{2},
\]
but this inequality is the limit of \eqref{eq:subparabolicineq} as $x \to 0$, and it can be verified in a similar way that $v$ satisfies \eqref{eq:subparabolic}, and is a subsolution even in this case.
 
 Now we consider the level set $v(x, t)=c$, $c<0$, and calculate
 \begin{align*}
 v(x, t)&=\frac{-1}{|\log{|t|}|^{\alpha+1}}\mathrm{e}^{-\frac{|x|^2}{4t}k}+\frac{1}{\log|\log{|t|}|} =c \\
  &\iff \frac{-1}{|\log{|t|}|^{\alpha+1}}\mathrm{e}^{-\frac{|x|^2}{4t}k} =c-\frac{1}{\log|\log{|t|}|} \\
&\iff \mathrm{e}^{-\frac{|x|^2}{4t}k} =|\log{|t|}|^{\alpha+1}\left(\frac{1}{\log|\log{|t|}|}-c\right) \\
 &\iff -\frac{|x|^2}{4t}k = (\alpha+1)\log|\log{|t|}|+\log\left(\frac{1}{\log|\log{|t|}|}-c\right), 
 \end{align*}
 or simply
 \begin{equation}
 \label{eq:smalldomain}
  |x|^2 =-4t\left(\frac{\alpha+1}{k}\log|\log{|t|}|+\frac{1}{k}\log\left(\frac{1}{\log|\log{|t|}|}-c\right)\right).
 \end{equation}
 Letting $\tilde{\Omega}$ denote the domain enclosed by \eqref{eq:smalldomain} and the hyperplane $t=t_0$, we have that for $c<0$, the function $v$ \eqref{eq:irregularitatsbarriare} is equal to $c$ on the curved portion of the boundary of $\tilde{\Omega}$. Also, $v$ tends to $0$ at the points converging to the origin along the $t$-axis . This shows that the origin is an irregular boundary point for $\tilde{\Omega}$.
 
 The inclusion $\tilde \Omega \subset \Omega$ requires that
 \[
 \frac{\alpha+1}{k}\log|\log{|t|}|+\frac{1}{k}\log\left(\frac{1}{\log|\log{|t|}|}-c\right) <(1+\epsilon)\log|\log|t||
 \]
 for small $|t|$. Fix $k$ close to 1 and $\alpha$ close to 0 so that
 \[
 \frac{\alpha+1}{k}<1+\frac{\epsilon}{2}.
 \]
 With this choice, $\delta_0$ in \eqref{eq:choosing delta} depends only on $\epsilon$. Thus we have to verify that
 \[
 \left(\frac{1}{\log|\log{|t|}|}+|c|\right)^{\frac{1}{k}} \leq |\log|t||^{\frac{\epsilon}{2}},
 \] 
 but this obviously holds for small $|t|$ since the left-hand side is bounded.
 Hence $\tilde{\Omega} \subset \Omega$ and $(0,0)$ is an irregular boundary point for $\Omega$ as well.
 
 This implies that the upper Perron solution $\overline{H}_f$ does not always attain the boundary values at $(0, 0)$.
 \end{proof}

 \appendix
\section{Barrier Families}
\label{sec:appa}
\noindent In this Section, a sub/supersolution is a viscosity sub/supersolution of 
\begin{equation}
\label{eq:infinity_heat}
u_t -\Delta_{\infty}u = u_t -\sum_{i, j=1}^{n}u_{x_i}u_{x_ix_j}u_{x_j} =0,
\end{equation}
the usual, non-normalised $\infty$- heat equation. Further, the upper and lower classes $\mathcal{U}_f$ and $\mathcal{L}_f$, and the Upper and Lower Perron solutions $\overline{H}_f$ and $\underline{H}_f$, now also refer to \eqref{eq:infinity_heat}. As before, the boundary function $f$ is assumed to be continuous.

Note that if $u$ is a radial function $u(x, t)=G(|x|, t)$, then \eqref{eq:infinity_heat} reads
\begin{equation*}
G_t(r, t)-G_r(r, t)^2G_{rr}(r, t)=0.
\end{equation*}
Also, if $u$ is a sub/supersolution to \eqref{eq:infinity_heat}, then $-u$ is a \\ super/subsolution.

The following definition is found in \cite{Bjorn2015}:
\begin{definition}
\label{def:stronk_barrier}
Let $\zeta_0 =(x_0, t_0) \in \partial \Omega$. A family of functions $\{w_j\}_j $ in $\Omega$ is called a \emph{barrier family} in $\Omega$ at the point $\zeta_0$ if, for every $j$,
\begin{enumerate}
\item $w_j>0$ and $w_j$ a supersolution in $\Omega$,
\item $\lim_{\zeta \to \zeta_0}w_j(\zeta)=0$,
\item for every $k=1, 2, \cdots,$ there is a $j$ so that
\begin{equation*}
\liminf_{\eta \to \zeta}w_j(\eta)\geq k
\end{equation*}
for all $\zeta \in \partial\Omega$ with $|\zeta-\zeta_0|\geq 1/k$.
\end{enumerate}
We also say that the family $\{w_j\}$ is a \emph{strong barrier family} in $\Omega$ if we also have that
\begin{enumerate}
	  \setcounter{enumi}{3}
\item $w_j$ is continuous in $\Omega$ and
\item there is a non-negative, continuous function $d$ with $d(z)=0$ if and only if $z=\zeta_0$, such that for every $k=1, 2, \cdots$ there is a $j=j(k)$ such that 
\begin{equation*}
w_j \geq kd
\end{equation*}
in $\Omega$.
\end{enumerate}
\end{definition}
Then we have the following result:
\begin{theorem}
The following are equivalent for $\zeta_0 \in \partial \Omega$:
\begin{enumerate}
\item $\zeta_0$ is regular for \eqref{eq:infinity_heat},
\item there is a barrier family at $\zeta_0$,
\item there is a \emph{strong} barrier family at $\zeta_0$.
\end{enumerate}
\end{theorem}
\begin{proof}
$(2) \implies (1).$

Assume that there exists a barrier family $\{w_j\}_j$ at $\zeta_0$. Since $f$ is continuous, we have that for every $\epsilon>0$ there exists a $\delta$ so that \\ $|f(\zeta)-f(\zeta_0)|<\epsilon$ whenever $|\zeta-\zeta_0|<\delta$. Using point (3) in the definition, we can thus find a large $j$ so that
\begin{equation*}
\liminf_{\eta \to \zeta}w_j(\eta) +\epsilon+f(\zeta_0) > f(\zeta)
\end{equation*}
for every $\zeta \in \partial \Omega.$ Also considering point (1) in the definition, we see that $w_j+\epsilon +f(\zeta_0) \in \mathcal{U}_f$. Hence
\begin{equation*}
\limsup_{\zeta \to \zeta_0}\overline{H}_f(\zeta) \leq \lim_{\zeta \to\zeta_0}w_j(\zeta)+\epsilon+f(\zeta_0) =\epsilon +f(\zeta_0),
\end{equation*}
by point (2) and the definition of the Upper Perron solution.

Noting that $-w_j$ is a subsolution, we can choose $j$ large enough so that
\begin{equation*}
\limsup_{\zeta \to \zeta_0}(-w_j)-\epsilon +f(\zeta_0) <f(\zeta), 
\end{equation*}
and hence $-w_j-\epsilon +f(\zeta_0)$ is in the Lower class, and
\begin{equation*}
\liminf_{\zeta \to \zeta_0}\overline{H}_f(\zeta) \geq \liminf_{\zeta \to \zeta_0}\underline{H}_f(\zeta) \geq -\epsilon +f(\zeta_0).
\end{equation*}

$(1) \implies (3)$. Assume that $\zeta_0$ is regular. We shall construct a \emph{strong} barrier family at $\zeta_0$. Define 
\begin{equation}
\label{eq:stronk_barrier}
\Psi_j(x, t) =j\alpha|x-x_0|^{\frac{4}{3}}+\beta j^m (t-t_0)^2,
\end{equation}
where $\alpha, \beta $ and $m$ are positive constants to be determined later. We show that $\Psi_j$ is a classical subsolution. 
\begin{align*}
\frac{\partial \Psi_j}{\partial t}-\Delta_{\infty}\Psi_j &= 2\beta j^m(t-t_0)-\frac{64j^3\alpha^3}{27} \\
&\leq 2\beta j^m \text{diam}(\Omega)-\frac{64j^3\alpha^3}{27}.
\end{align*}
Choosing, for example, $\beta =\frac{1}{2\text{diam}(\Omega)}$, $\alpha=1$ and $m=3$, we get
\begin{equation*}
\frac{\partial \Psi_j}{\partial t}-\Delta_{\infty}\Psi_j \leq j^3(1-\frac{64}{27})=-\frac{37}{27}j^3 <0,
\end{equation*}
and $\Psi_j$ is a subsolution.
Further, choosing
\begin{equation*}
d(x, t) =|x-x_0|^{\frac{4}{3}}+\frac{1}{2\text{diam}(\Omega)}(t-t_0)^2, 
\end{equation*}
we see that $\Psi_j(x, t)\geq j^3d(x, t)$.

Now, setting $w_j=\underline{H}_{\Psi_j}$, we have that $\{w_j\}_j$ is a strong barrier family at $(x_0, t_0)$. Indeed, (2) in Definition \ref{def:stronk_barrier} follows since we assumed $(x_0, t_0)$ is regular, and by the definition of the Lower Perron solution we have that 
\begin{equation*}
w_j \geq \Psi_j \geq j^3d,
\end{equation*}
and so (5) holds. Hence $\{w_j\}_j$ is a strong barrier family. 

$(3) \implies (2)$ is trivial.
\end{proof}
\section*{Acknowledgements}
\noindent The author would like to thank Peter Lindqvist for his help and guidance, as well as Bernd Kawohl for helpful input and several suggestions. Jana Björn and Vesa Julin are thanked for discovering a flaw in the proof of the Petrovsky criterion, and  the anonymous referee is thanked for his constructive feedback.
\bibliographystyle{alpha}
\bibliography{/Users/nikubo74/Desktop/PhD/refs} 

	\end{document}